\documentclass [11pt]{amsart}
\usepackage {amsmath, amssymb, amscd, mathrsfs, hyperref, graphicx, color, slashed, enumerate, url}
\usepackage[text={6in,8.5in},centering,letterpaper,dvips]{geometry}
\usepackage[all,cmtip]{xy}

\newtheorem {theorem}{Theorem}[section]

\newtheorem {question}[theorem]{Question}

\theoremstyle{remark}

\newtheorem {example}[theorem]{Example}

\DeclareFontFamily{U}{mathx}{\hyphenchar\font45}
\DeclareFontShape{U}{mathx}{m}{n}{
      <5> <6> <7> <8> <9> <10>
      <10.95> <12> <14.4> <17.28> <20.74> <24.88>
      mathx10
      }{}
\DeclareSymbolFont{mathx}{U}{mathx}{m}{n}
\DeclareFontSubstitution{U}{mathx}{m}{n}
\DeclareMathAccent{\widecheck}{0}{mathx}{"71}

\def\polhk#1{\setbox0=\hbox{#1}{\ooalign{\hidewidth
    \lower1.5ex\hbox{`}\hidewidth\crcr\unhbox0}}}  
\def\Z {{\mathbb{Z}}}
\def\R {{\mathbb{R}}}
\def\C {{\mathbb{C}}}
\def\Q {{\mathbb{Q}}}

\def\tH{\tilde{H}}

\def\cp{\mathbb{CP}}

\def\F {\mathbb{F}}
\def\del {\partial}

\def\tr {\operatorname{tr}}
\def\Sym{\operatorname{Sym}}
\def\Ta {\mathbb{T}_{\alpha}}
\def\Tb {\mathbb{T}_{\beta}}

\def\S{\mathscr{S}}

\DeclareMathOperator{\id}{\operatorname{id}}
\def\ind{\operatorname{ind}}
\DeclareMathOperator{\rk}{\operatorname{rk}}

\def\Sq{\operatorname{Sq}}
\def\End{\operatorname{End}}

\def\To {\longrightarrow}

\def\spinc{\operatorname{Spin}^c}
\def\sutwo {\operatorname{SU}(2)}
\def\sothree {\operatorname{SO}(3)}

\def\nbhd{\operatorname{nbhd}}
\def\P{\mathcal{P}}

\def\link{\mathit{link}}

\def\Kh{\widetilde{\mathit{Kh}}}
\def\swf{\operatorname{SWF}}
\def\swfh{\mathit{SWFH}}
\def\swfk{\mathit{SWFK}}
\def\tK{\tilde{K}}

\def\U{\operatorname{U}}

\def\pin {\operatorname{Pin}(2)}

\newcommand{\HMto}{\widecheck{\mathit{HM}}}
\newcommand{\HMfrom}{\widehat{\mathit{HM}}}
\newcommand{\HMbar}{\overline{\mathit{HM}}}
\def\HF{\mathit{HF}}
\def\HFp {\HF^+}
\def\HFhat {\widehat{\HF}}
\def\H {\mathbb{H}}

\newcommand{\s}{\mathfrak{s}}

\def\pt {\operatorname{pt}}

\begin{document}

\title[Floer theory and its topological applications]{Floer theory and its topological applications}

\author[Ciprian Manolescu]{Ciprian Manolescu}
\address {Department of Mathematics, UCLA, 520 Portola Plaza\\ Los Angeles, CA 90095}
\email {cm@math.ucla.edu}

\begin{abstract}
We survey the different versions of Floer homology that can be associated to three-manifolds. We also discuss their applications, particularly to questions about surgery, homology cobordism, and four-manifolds with boundary. We then describe Floer stable homotopy types, the related $\pin$-equivariant Seiberg-Witten Floer homology, and its application to the triangulation conjecture.
\end {abstract}

\maketitle

\section{Introduction}

In finite dimensions, one way to compute the homology of a compact, smooth manifold is by Morse theory. Specifically, we start with a a smooth function $f: X \to \R$ and a Riemannian metric $g$ on $X$. Under certain hypotheses (the Morse and Morse-Smale conditions), we can form a complex $C_*(X, f, g)$ as follows: The generators of $C_*(X, f, g)$ are the critical points of $f$, and the differential is given by
\begin{equation}
\label{eq:delmorse}
 \del x = \sum_{\{y \mid \ind(x)-\ind(y)=1\}} n_{xy}\cdot y,
 \end{equation}
where $n_{xy} \in \Z$ is a signed count of the downward gradient flow lines of $f$ connecting $x$ to $y$. The quantity $\ind$ denotes the index of a critical point, that is, the number of negative eigenvalues of the Hessian (the matrix of second derivatives of $f$) at that point. The homology $H_*(X, f, g)$ is called {\em Morse homology}, and it turns out to be isomorphic to the singular homology of $X$ \cite{WittenMorse, FloerMorse, BottMorse}. 

Floer homology is an adaptation of Morse homology to infinite dimensions. It applies to certain classes of infinite dimensional manifolds $X$ and functions $f: X \to \R$, where at critical points of $f$ the Hessian has infinitely many positive and infinitely many negative eigenvalues. Although one cannot define the index $\ind(x)$ as an integer, one can make sense of a relative index $\mu(x, y) \in \Z$ which plays the role of $\ind(x)-\ind(y)$ in the formula \eqref{eq:delmorse}. Then, one can define a complex just as above, and the resulting homology is called {\em Floer homology}. This is typically not isomorphic to the homology of $X$, but rather encodes new information---usually about a finite dimensional manifold from which $X$ was constructed.

Floer homology appeared first in the context of symplectic geometry \cite{FloerFirst, FloerIndex, FloerAction, FloerLagrangian}. In the version called Hamiltonian Floer homology, one considers a compact symplectic manifold $(M, \omega)$ together with a $1$-periodic Hamiltonian function $H_t$ on $M$. From this one constructs the infinite dimensional space $X=\mathcal{L}M$ of contractible loops in $M$, together with a symplectic action functional $\mathcal{A}: X \to \R$. The critical points of $\mathcal{A}$ are periodic orbits of the Hamiltonian flow, and the gradient flow lines correspond to pseudo-holomorphic cylinders in $M$. Hamiltonian Floer homology can be related to the homology of $M$; the main applications of this fact are the proofs of the Arnol'd conjecture \cite{FloerArnold, OnoArnold, HoferSalamonArnold, FukayaOnoArnold, RuanArnold, LiuTian}.

A more general construction in symplectic geometry is Lagrangian Floer homology. Given two Lagrangians $L_0, L_1$ in a symplectic manifold $(M, \omega)$, one considers the space of paths
$$ \P(M; L_0, L_1) = \{ \gamma: [0,1] \to M \mid \gamma(0) \in L_0, \gamma(1) \in L_1 \}$$
together with a certain functional. The critical points correspond to the intersections of $L_0$ with $L_1$, and the gradient flows to pseudo-holomorphic disks with half of the boundary on $L_0$ and half on $L_1$. Under some technical assumptions, the resulting Lagrangian Floer complex is well-defined, and its homology can be used to give bounds on the number of intersection points $x \in L_0 \cap L_1$. Since its introduction in \cite{FloerLagrangian}, Lagrangian Floer homology has developed into one of the most useful tools in the study of symplectic manifolds; see \cite{FOOO1, FOOO2, SeidelBook, OhBook} for a few books devoted to this subject. Hamiltonian Floer homology can be viewed as a particular case of Lagrangian Floer homology: Given $M$ and a Hamiltonian $H_t$ producing a time $1$ diffeomorphism $\psi$, the Lagrangian pair that we need to consider is given by the diagonal and the graph of $\psi$ inside $M \times M$. 

Apart from symplectic geometry, the other area where Floer homology has been very influential is low dimensional topology. There, Floer homology groups are associated to a closed three-manifold $Y$ (possibly of a restricted form, and equipped with certain data). The first construction of this kind was the instanton homology of Floer \cite{FloerInstanton}, where the infinite dimensional space $X$ is the space of $\sutwo$ (or $\sothree$) connections on $Y$ (modulo the gauge action), and $f$ is the Chern-Simons functional. This construction has an impact in four dimensions: the relative Donaldson invariants of four-manifolds with boundary take values in instanton homology.

In this paper we survey Floer theory as it is relevant to low dimensional topology. We will discuss four types of Floer homology that can be associated to a three-manifold (each coming with its own different sub-types):
\begin{enumerate}[(a)] \itemsep5pt
\item Instanton homology;
\item Symplectic instanton homology;\footnote{The terminology ``symplectic instanton homology'' is not yet standard. We use it to mean the different kinds of Lagrangian Floer homology that are meant to recover instanton homology. By the same token, Heegaard Floer homology could be called ``symplectic monopole Floer homology.''}
\item Monopole Floer homology;
\item Heegaard Floer homology.
\end{enumerate}

The types (a) and (c) above are constructed using gauge theory. In (a), the gradient flow lines of the Chern-Simons functional are solutions of the anti-self-dual Yang-Mills equations on $\R \times Y$, whereas in (c) we consider the Chern-Simons-Dirac functional, whose gradient flow lines are solutions to the Seiberg-Witten equations on $\R \times Y$. Different definitions of monopole Floer homology were given in \cite{MarcolliWang, Spectrum, KMBook, FroyshovSW}.

The types (b) and (d) are symplectic replacements for (a) and (c), respectively. Their construction starts with a decomposition of the three-manifold $Y$ into a union of two handlebodies glued along a surface $\Sigma$. To $\Sigma$ we associate a symplectic manifold $M(\Sigma)$: in (b), this is a moduli space of flat connections on $\Sigma$, whereas in (d) it is a symmetric product of $\Sigma$ (which can be interpreted as a moduli space of vortices). To the two handlebodies we then associate Lagrangians $L_0, L_1 \subset M(\Sigma)$, and their Lagrangian Floer homology is the desired theory. This kind of construction was suggested by Atiyah \cite{AtiyahFloer}, and the equivalence of (a) and (b) came to be known as the Atiyah-Floer conjecture. In the monopole context, the analogous construction was pursued by Ozsv\'ath and Szab\'o, who developed Heegaard Floer homology in a series of papers \cite{HolDisk, HolDiskTwo, HolDiskFour, AbsGraded}. The equivalence of (c) and (d) was recently established \cite{KLT1, KLT2, KLT3, KLT4, KLT5, CGH1, CGH2, CGH3}.

We will describe the four different Floer homologies (a)-(d) in each of the Sections \ref{sec:inst} through \ref{sec:HF}, respectively. Throughout, we will give a sample of the applications of these theories to questions in low dimensional topology.

In Section~\ref{sec:swf} we will turn to the question of constructing Floer generalized homology theories, such as Floer stable homotopy. We will discuss how this was done in the monopole setting in \cite{Spectrum}. A by-product of this construction was the definition of $\pin$-equivariant Seiberg-Witten Floer homology, whose main application is outlined in Section~\ref{sec:pin2}: the disproof of the triangulation conjecture for manifolds of dimension $\geq 5$.

\section{Instanton homology}
\label{sec:inst}

The first application of gauge theory to low dimensional topology was Donaldson's diagonalizability theorem:

\begin{theorem}[Donaldson \cite{Donaldson}]
Let $W$ be a simply connected, smooth, closed $4$-manifold. If the intersection form of $W$ is definite, then it can be diagonalized over $\Z$.
\end{theorem}

The proof uses a study of the anti-self-dual Yang-Mills equations on $W$:
\begin{equation}
\label{eq:YM}
 \star F_A = -F_A,
\end{equation}
where $A$ is a connection in a principal $\sutwo$ bundle $P$ on $W$, and $F_A$ denotes the curvature of $A$. 

Later, Donaldson introduced his polynomial invariants of $4$-manifolds \cite{DonaldsonPolynomials}, which are a signed count of the solutions to \eqref{eq:YM}, modulo the action of the gauge group $\Gamma(\End P)$. These have numerous other applications to four-dimensional topology. 

Instanton homology is a (relatively $\Z/8$-graded) Abelian group $I_*(Y)$ associated to a closed $3$-manifold $Y$ (with some restrictions on $Y$; see below). The main motivation behind the construction of instanton homology is to develop cut-and-paste methods for computing the Donaldson invariants. Roughly, if we have a decomposition of a closed $4$-manifold $W$ as
$$ W = W_1 \cup_Y W_2,$$
then one can define relative Donaldson invariants $D(W_1) \in I_*(Y)$ and $D(W_2) \in I_*(-Y)$ such that the invariant of $W$ is obtained from $D(W_1)$ and $D(W_2)$ under a natural pairing map
$$ I_*(Y) \otimes I_*(-Y) \to \Z.$$

More generally, instanton homology fits into a topological quantum field theory (TQFT). Given a $4$-dimensional cobordism $W$ from $Y_0$ to $Y_1$, we get a map
$$ D(W) : I_*(Y_0) \to I_*(Y_1)$$
which is functorial under composition of cobordisms.

As mentioned in the introduction, to define instanton homology we consider an $\sutwo$ bundle $P$ over $Y$ (in fact, there is a unique such bundle, the trivial one) and form the infinite dimensional space
$$ X = \{ \text{connections on } P \}/ \text{gauge},$$
with the Chern-Simons functional
$$ \operatorname{CS}: X \to \R/ \Z, \ \ \ \operatorname{CS}(A)= \frac{1}{8\pi^2}\int_Y \tr(A \wedge dA + A \wedge A \wedge A).$$
We then define a Morse complex for $\operatorname{CS}$. Its generators are connections $A$ with $F_A=0$, modulo gauge; these can be identified with representations $\pi_1(Y) \to \sutwo$, modulo the action of $\sutwo$ by conjugation. Further, the differential counts gradient flow lines, which can be re-interpreted as solutions to \eqref{eq:YM} on the cylinder $\R \times Y$. The homology of this complex is $I_*(Y)$.

The above is only a rough sketch of the construction. There are many caveats, such as:
\begin{enumerate}[(i)] \itemsep5pt
\item To define the gradient we also need to choose a Riemannian metric $g$ on $Y$. (However, the resulting instanton homology will be independent of $g$.)

\item The Chern-Simons functional has to be suitably perturbed to achieve transversality.

\item One needs to distinguish between irreducible connections (those with trivial stabilizer under the gauge group) and reducible connections. In Floer's original theory \cite{FloerInstanton}, one only counts irreducibles. However, to have $\del^2=0$ in the complex, we need to make sure the interaction with the reducibles is negligible. This happens provided that either:
\begin{enumerate}[(a)] 
\item $Y$ is a homology sphere; or
\item Instead of the $\sutwo$ bundle we use a non-trivial $\sothree$ bundle, satisfying an admissible condition (so that there are no reducibles). Such bundles only exist for $b_1(Y) > 0$. 
\end{enumerate}
There are also versions of Floer homology that involve the reducible; see \cite{DonaldsonBook}.
\end{enumerate}

Point (iii) above says that Floer's instanton homology $I_*(Y)$ is only defined in cases (a) and (b), i.e. for homology spheres and for admissible bundles. If one wants a consistent theory for all $3$-manifolds $Y$, one way to produce it is to take a connected sum with a fixed $3$-manifold (such as $T^3$) that is equipped with an admissible bundle $P_0$. The resulting group
\begin{equation}
\label{eq:isharp}
 I^{\#}(Y) := I(Y \# T^3, P \# P_0)
 \end{equation}
is called {\em framed instanton homology}; cf. \cite{KMinst1}.

We now turn to a few applications of instanton homology.

Given the TQFT structure, it is not surprising that many applications have to do with four-manifolds with boundary and cobordisms. In particular, let us consider the three-dimensional homology cobordism group
\begin{equation}
\label{eq:theta}
\Theta^H_3 = \{ \text{oriented homology 3-spheres}\}/ \sim
\end{equation}
where $Y_0 \sim Y_1 \iff$  there exists a smooth, compact, oriented $4$-manifold $W$ with $\del W = (-Y_0) \cup Y_1$ and $H_1(W; \Z) = H_2(W; \Z)=0.$ Addition in $\Theta^H_3$ is given by connected sum, the inverse is given by reversing the orientation, and $S^3$ is the zero element.  

The first information about $\Theta^H_3$ came from the Rokhlin homomorphism \cite{Rokhlin, EellsKuiper}:
\begin{equation}
\label{eq:mu}
\mu: \Theta^H_3 \to \Z/2, \ \  \mu(Y) = {\sigma(W)}/{8} \pmod {2},
\end{equation}
where $W$ is any compact, spin $4$-manifold with boundary $Y$. One can prove that the value of $\mu$ depends only on $Y$, not on $W$. The homomorphism $\mu$ can be used to show that $\Theta^H_3$ is non-trivial: For instance, the Poincar\'e sphere $P$ bounds the $E_8$ plumbing (of signature $-8$), so $\mu(P)=1$.

The structure of $\Theta^H_3$ is still not completely understood. Most of what we know comes from gauge theory. Using the Yang-Mills equations (albeit without referencing Floer homology directly), Furuta and Fintushel-Stern proved that $\Theta^H_3$ is infinitely generated \cite{FurutaHom, FSinstanton}. Using the $\sutwo$-equivariant structure on instanton homology, Fr{\o}yshov \cite{FroyshovYM} defined a surjective homomorphism
$$ h:  \Theta^H_3 \to \Z.$$
This implies the following:

\begin{theorem}[Fr{\o}yshov \cite{FroyshovYM}] 
\label{thm:summand}
The group $\Theta^H_3$ has a $\Z$ summand. 
\end{theorem}

Moreover, Fr{\o}yshov found the following generalization of Donaldson's theorem to $4$-manifolds with boundary:
\begin{theorem}[Fr{\o}yshov \cite{FroyshovYM}] 
\label{thm:bdry}
If a homology sphere $Y$ bounds a smooth, compact oriented $4$-manifold with negative definite intersection form, then $h(Y) \geq 0$. This inequality is strict if the intersection form is not diagonal over the integers.
\end{theorem}

Another celebrated application of instanton homology is the proof of Property P for knots. Given a knot $K \subset S^3$ and relatively prime integers $p, q$, the result of $p/q$ surgery on $K$ is the three-manifold
$$ S^3_{p/q}(K)= (S^3 - \nbhd(K)) \cup_{T^2} (S^1 \times D^2),$$
where the gluing along the torus is done such that the meridian $\{1\} \times D^2$ is taken to a simple closed curve in the homology class $p[\mu]+q[\lambda]$. (Here, $\lambda$ and $\mu$ are the longitude and the meridian of the knot.)

A theorem of Lickorish and Wallace \cite{Lickorish, Wallace} says that every closed $3$-manifold can be obtained by surgery on a collection of knots in $S^3$. Those manifolds obtained by surgery on a single knot form an interesting class. Before Perelman's proof of the Poincar\'e conjecture \cite{Perelman1, Perelman2, Perelman3}, as a first step towards the conjecture, one could ask whether any counterexamples can be obtained by surgery on a knot. Since Gordon and Luecke had shown that $S^3_{p/q}(K)=S^3$ only when $K$ is the unknot \cite{GordonLuecke}, that  question can be rephrased as follows: Does every non-trivial knot $K \subset S^3$ have property P, i.e.,  do we have $$\pi_1(S^3_{p/q}(K)) \neq 1$$ for all $p/q \in \Q$? For $p/q \neq \pm 1$, this was established in \cite{CGLS}. The remaining case $p/q =\pm 1$ was completed by Kronheimer and Mrowka in 2004 (independently of Perelman's work):

\begin{theorem}[Kronheimer-Mrowka \cite{KMPropP}]
\label{thm:propP}
If a homotopy $3$-sphere $Y$ is obtained by $\pm 1$ surgery on a knot $K \subset S^3$, then $K$ is the unknot (and hence $Y$ is the three-sphere). 
\end{theorem}

This result builds on the work of many mathematicians; it uses results from symplectic and contact geometry, as well as gauge theory. Instanton homology enters the picture through the connection between the generators of the Floer complex (flat connections) and representations $\pi_1(Y) \to \sutwo$. The final step in the proof of Theorem~\ref{thm:propP} is to show that $I_*(S^3_{\pm 1}(K)) \neq 0$, which implies the existence of a non-trivial representation $\pi_1(S^3_{\pm 1}(K))  \to \sutwo$, and hence the non-vanishing of the fundamental group.

Here is another application of instanton homology to knot theory. Recall that to a knot $K \subset S^3$ we can associate the Jones polynomial 
$$V_K(t) \in \Z[t, t^{-1}].$$ A natural question is whether the Jones polynomial detects the unknot $U$, that is, if $V_K(t)=V_U(t)=1$, do we have $K=U$? This is still open, but a ``categorified'' version of this question has been answered. Specifically, in \cite{Khovanov, KhovanovPatterns}, Khovanov defined (combinatorially) a bi-graded homology theory for knots
$$ \Kh(K) = \bigoplus_{i, j\in \Z} \Kh^{i,j}(K)$$
such that its Euler characteristic gives the Jones polynomial:
$$ \sum_{i,j \in \Z} (-1)^i t^j \rk \Kh^{i,j}(K) = V_K(t).$$
It turns out that Khovanov homology detects the unknot:

\begin{theorem}[Kronheimer-Mrowka \cite{KMUnknot}]
If a knot $K \subset S^3$ has $\Kh(K)=\Kh(U)$, then $K=U$.
\end{theorem}

The proof uses a version of instanton homology for knots, $I^{\natural}(K)$. There is a spectral sequence relating $\Kh(K)$ to $I^{\natural}(K)$, which implies a rank inequality between the two theories, $\rk \Kh(K) \geq \rk I^{\natural}(K)$. Using sutured decompositions of the knot complement, it can be shown that $\rk I^{\natural}(K) \geq 1$, with equality if and only if $K$ is the unknot; see \cite{KMsutures, KMUnknot}. In turn, this implies the corresponding result for $\Kh$.

\section{Symplectic instanton homology}
\label{sec:sinst}

A Heegaard splitting of a closed $3$-manifold $Y$ is a decomposition
\begin{equation}
\label{eq:Heegaard}
Y = U_0 \cup_{\Sigma} U_1,
\end{equation}
where $\Sigma$ is a surface of genus $g$ and $U_0$, $U_1$ are handlebodies. Given such a splitting (which can be found for any $Y$), we consider the moduli space $M(\Sigma)$ of $\sutwo$ flat connections over $\Sigma$, modulo gauge. We can identify it with the representation space
$$ \{\pi_1(\Sigma) \to \sutwo\} / \sutwo.$$
The handlebodies $U_i$ ($i=0,1$) produce subspaces $L_i \subset M(\Sigma)$, corresponding to representations that extend to $\pi_1(U_i)$. 

The Atiyah-Floer conjecture \cite{AtiyahFloer} states that there should be an isomorphism:
\begin{equation}
\label{eq:AF}
 I_*(Y) \cong HF_*(L_0, L_1),
 \end{equation}
where the left hand side is instanton homology, and the right hand side is Lagrangian Floer homology inside $M(\Sigma)$. The idea behind the conjecture is to deform the metric on $Y$ by inserting a long cylinder of the form $[-T, T] \times \Sigma$ in the middle of the decomposition \eqref{eq:Heegaard}. As $T \to \infty$, we expect the flat connections on $Y$ to ``localize'' to intersection points $L_0 \cap L_1 \subset M(\Sigma)$, and the ASD Yang-Mills equations on $\R \times Y$ to turn into the nonlinear Cauchy-Riemann equations on $M(\Sigma)$ that define pseudo-holomorphic curves.

The first difficulty with \eqref{eq:AF} is that the right hand side (which we call {\em symplectic instanton homology}) is not well-defined. This is because the moduli space $M(\Sigma)$ and the Lagrangians $L_0, L_1$ are singular (at the points corresponding to reducible connections). Still, by tweaking the definition in various ways, one can define the right hand side in certain settings:
\begin{itemize} \itemsep5pt
\item Dostoglou-Salamon \cite{DostoglouSalamon} considered $\U(2)$ connections in a non-trivial bundle over $\Sigma$. The resulting moduli space $M'(\Sigma)$ is smooth, and using it they formulate (and then prove) a version of \eqref{eq:AF} for mapping tori;
\item
Salamon-Wehrheim \cite{SalamonWehrheim} defined a Lagrangian Floer homology in infinite dimensions, as the first step in a program for establishing \eqref{eq:AF} for homology spheres $Y$;
\item
Wehrheim-Woodward \cite{WehrheimWoodwardFFT} developed Lagrangian Floer homology in $M'(\Sigma)$ further. They define the right hand side of \eqref{eq:AF} whenever $Y$ is equipped with an admissible bundle (with no reducibles). In particular, by taking connected sum with a torus, they can define a ``framed'' version of symplectic instanton homology, which is conjecturally the same as $I^{\#}(Y)$ from \eqref{eq:isharp};
\item
Another definition of framed symplectic instanton homology was proposed by Manolescu-Woodward \cite{MWextended}. This is based on doing Lagrangian Floer homology inside a (smooth) extended moduli space of $\sutwo$ connections, whose symplectic quotient is $M(\Sigma)$.  
\end{itemize}

For recent progress towards the Atiyah-Floer conjecture for admissible bundles (i.e., with no reducibles), see \cite{Duncan, LipyanskiyAF}.

Symplectic instanton homology has not yet produced any significant topological applications. Nevertheless, it motivated the analogous construction in the monopole setting, which led to the development of Heegaard Floer homology (cf. Section~\ref{sec:HF} below). 

\section{Monopole Floer homology}
\label{sec:mon}

Apart from the ASD Yang-Mills equations, the other main input that gauge theory provides for the study of $4$-manifolds is the Seiberg-Witten (or monopole) equations \cite{SW1, Witten}:
 \begin{equation}
 \label{eq:sw}
 F_A^+ = \sigma(\phi), \ \ D_A \phi =0.
 \end{equation}
 These are associated to a $4$-manifold $W$ equipped with a $\spinc$ structure $\s$, with spinor bundles $S^+, S^-$. In these equations, $A$ is a $\spinc$ connection, $\phi$ is a section of $S^+$, $D_A$ is the Dirac operator associated to $A$, and $\sigma$ is a certain quadratic expression in $\phi$. The signed count of solutions to \eqref{eq:sw} gives the Seiberg-Witten invariant of the pair $(W, \s)$.

Monopole Floer homology is obtained from the Seiberg-Witten equations similarly to how instanton homology is obtained from the Yang-Mills equations. Given a three-manifold $Y$ with a $\spinc$ structure $\s$, we consider an infinite dimensional configuration space of connection-spinor pairs $(A, \phi)$, modulo gauge. (Here, $A$ is a connection in a $\U(1)$, rather than in an $\sutwo$ or $\sothree$ bundle.) The configuration space is equipped with the Chern-Simons-Dirac functional CSD, given by
$$ \operatorname{CSD}(A, \phi) = -\frac{1}{8} \int_Y (A^t - A_0^t) \wedge (F_{A^t} + F_{A_0^t}) + \frac{1}{2} \int_Y \langle D_A \phi, \phi\rangle \ d \! \operatorname{vol}.$$
Here, $A_0$ is a fixed base connection, and the superscript $t$ denotes the induced connections in the determinant bundle. 

The Floer homology associated to CSD is monopole Floer homology. There are several difficulties that need to be overcome to make this definition precise. As in the instanton case, the main problem is the presence of reducible connections. In their monograph on the subject \cite{KMBook}, Kronheimer and Mrowka deal with this by considering a (real) blow-up of the configuration space. They succeed in defining monopole Floer versions (in three versions: $\HMto$, $\HMfrom$, $\HMbar$) for all pairs $(Y, \s)$. 

When $Y$ is a rational homology sphere, alternate constructions of monopole Floer homology have been proposed by Marcolli-Wang \cite{MarcolliWang}, Manolescu \cite{Spectrum}, and Fr{\o}yshov \cite{FroyshovSW}.

Monopole Floer homology can be applied to questions about homology cobordism and four-manifolds with boundary, in a manner similar to instanton homology. In particular, one can define a surjective homomorphism
$$ \delta: \Theta^3_H \to \Z$$ 
and give a proof of Theorem~\ref{thm:summand} using monopoles \cite{FroyshovSW, KMBook}. The definition of $\delta$ uses the $\Z[U]$-module structure on $\HMto$, which comes from the $S^1$-equivariance of the equations, since $H^*_{S^1}(\pt)=H^*(\cp^{\infty})=\Z[U].$ Precisely, we have:
\begin{equation}
\label{eq:delta}
 \delta(Y) = \tfrac{1}{2} \min \{r \mid \exists \ x \neq 0 \text{ s.t. } \forall l, \ x \in \operatorname{im}(U^l: \HMto_{r+2l}(Y) \to \HMto_r(Y)) \}.
 \end{equation}

If $Y$ is an integral homology sphere and $W$ is a negative-definite $4$-manifold with boundary $Y$, we have 
\begin{equation}
\label{eq:Fineq}
 c^2 + \operatorname{rk}(H^2(W; \Z)) \leq 8\delta(Y),
 \end{equation}
for any characteristic vector $c \in H_2(W; \Z)/$torsion, i.e., a vector such that $c \cdot v \equiv v \cdot v \pmod{2}$ for all $v \in H_2(W; \Z)/$torsion. This implies the analog of Theorem~\ref{thm:bdry}, that $\delta(Y) \geq 0$.

One advantage of monopole Floer homology (over instanton homology) is its closer relation to geometric structures on $3$-manifolds, such as embedded surfaces, taut foliations, and contact structures \cite{KMcontact}. (Inspired by the monopole case, similar connections were later proved to exist for instanton homology as well, but in a more roundabout way: using sutured decompositions; see \cite{KMsutures, BaldwinSivek}.) By exploiting the relation of monopole Floer homology to taut foliations, one can prove:

\begin{theorem}[Kronheimer-Mrowka-Ozsv\'ath-Szab\'o \cite{KMOS}]
\label{thm:KMOS}
Suppose $K \subset S^3$ is a knot such that there is an orientation-preserving diffeomorphism $S^3_r(K) \cong S^3_r(U)$, for some $r \in \Q$. Then $K$ is the unknot $U$.
\end{theorem}
 
An important ingredient in the proof of Theorem~\ref{thm:KMOS} are exact triangles that relate the Floer homologies of different surgeries on $K$. This allows one to reduce the argument to studying the Floer homology of $0$-surgeries. One then uses a non-vanishing result for the Floer homology of manifolds admitting taut foliations (such as $S^3_0(K)$ for $K \neq U$). 
 
Another celebrated application of monopole Floer homology is Taubes' solution to the Weinstein conjecture in dimension three:

\begin{theorem}[Taubes \cite{TaubesWeinstein}]
Let $Y$ be a closed $3$-manifold equipped with a contact form, and let $R$ be the associated Reeb vector field. Then $R$ has at least one periodic orbit.
\end{theorem}

The idea is to use the non-vanishing of monopole Floer homology to produce solutions to the Seiberg-Witten equations on $Y$. Then, one deforms these equations so that in the limit, the spinor is close to zero only on a set that approximates the periodic orbits of the Reeb vector field $R$.

\section{Heegaard Floer homology}
\label{sec:HF}

The definition of Heegaard Floer homology \cite{HolDisk} starts with a Heegaard splitting $Y = U_0 \cup_{\Sigma} U_1$, just as in \eqref{eq:Heegaard}. We then do Lagrangian Floer homology on a symplectic manifold associated to the surface $\Sigma$. In the case of symplectic instanton homology discussed in Section~\ref{sec:sinst}, the symplectic manifold was a moduli space of flat connections on $\Sigma$; these flat connections are solutions to a two-dimensional reduction of the ASD Yang-Mills equations. In the Heegaard Floer setting, we use instead the vortex equations on $\Sigma$, which are a reduction of the Seiberg-Witten equations. The moduli spaces of vortices are symmetric products of $\Sigma$. It is most convenient to consider the $g$th symmetric product, where $g$ is the genus of $\Sigma$:
$$ \Sym^g \Sigma = (\Sigma \times \dots \times \Sigma)/ \mathfrak{S}_g.$$
Here, we take the Cartesian product of $g$ copies of $\Sigma$, and then divide by the natural action of the symmetric group $\mathfrak{S}_g$.

To construct Lagrangians in $\Sym^g(\Sigma)$, pick simple closed curves $\alpha_1, \dots, \alpha_g \subset \Sigma$ that are homologically linearly independent in $\Sigma$, and bound disks in the handlebody $U_0$; pick also similar curves $\beta_1, \dots, \beta_g$ that bound disks in $U_1$. The set of data $(\Sigma; \alpha_1, \dots, \alpha_g; \beta_1, \dots, \beta_g)$ is called a {\em Heegaard diagram} for $Y$. Consider the tori
$$ \Ta = \alpha_1 \times \dots \times \alpha_g, \ \ \Tb = \beta_1 \times \dots \times \beta_g \subset \Sym^g(\Sigma).$$

Heegaard Floer homology is the Lagrangian Floer homology of $\Ta$ and $\Tb$ inside $\Sym^g(\Sigma)$. To get the full power of the theory, one picks a basepoint $z \in \Sigma$ (away from the $\alpha$ and $\beta$ curves) and then keeps track of the relative homotopy classes of pseudo-holomorphic disks through their intersection with the divisor  
$$ \{z \} \times \Sym^{g-1}(\Sigma) \subset \Sym^g(\Sigma).$$

This way one obtains three versions of Heegaard Floer homology, denoted $\HFp, \HF^-$ and $\HF^{\infty}$. They are modules over the ring $\Z[U]$, and correspond to the monopole Floer homologies $\HMto, \HMfrom$ and $\HMbar$, respectively. In this section we will focus on $\HFp$. By setting $U=0$ in the chain complex for $\HFp$ and then taking homology, one obtains a somewhat weaker theory denoted $\HFhat$, which is the Lagrangian Floer homology of $\Ta$ and $\Tb$ inside $\Sym^g(\Sigma - \{ z \}).$

Just as monopole Floer homology, Heegaard Floer homology decomposes according to the $\spinc$ structures on $Y$. For example:
$$ \HFp(Y) = \bigoplus_{\s \in \spinc} \HFp(Y, \s).$$

Among the Floer homologies of $3$-manifolds, Heegaard Floer homology is the most computationally tractable:
\begin{itemize} \itemsep5pt

\item The generators of the Heegaard Floer chain complex are $n$-tuples of intersection points between $\Ta$ and $\Tb$, so they can be easily read from a Heegaard diagram;

\item There are exact triangles relating $\HFp$ of different surgeries on a null-homologous knot, in an arbitrary $3$-manifold. Using these triangles, one can inductively compute $\HFp$ for large classes of plumbed manifolds, such as all Seifert fibered rational homology spheres \cite{Plumbed, Nemethi}. More generally, $\HFp$ for all Seifert fibered manifolds was computed in \cite{RatSurg};

\item By iterating the exact triangles, one can study $\HFp$ of the double branched cover of $S^3$ over a knot $K$, and relate it to the Khovanov homology of $K$ \cite{BrDCov}. In particular, one can explicitly calculate $\HFp$ of double branched covers over alternating knots;

\item There are surgery formulas that express $\HFp$ of a surgery on a knot in terms of a Floer complex associated to the knot \cite{IntSurg, RatSurg}. The knot Floer complex was defined in \cite{RasmussenThesis, Knots}, and knot Floer homology has many applications of its own; see \cite{HFKsurvey} for a survey;

\item The knot Floer complex of a knot (or link) in $S^3$ admits a combinatorial description in terms of grid diagrams \cite{MOS};

\item Using a special class of diagrams called nice, one can find a combinatorial description of $\HFhat(Y)$ for any $3$-manifold $Y$ \cite{SarkarWang};

\item There is also a surgery formula for links \cite{LinkSurg}. This expresses the Heegaard Floer homology of an integral surgery on a link in terms of Floer data associated to the link and its sublinks;

\item By combining the link surgery formula with the grid diagram technique for links in $S^3$, one arrives at a combinatorial description of $\HFp$ and $\HFhat$ for all $3$-manifolds. One also gets such a description for the related mixed invariants of $4$-manifolds from \cite{HolDiskFour} (analogues of the Seiberg-Witten invariants). See \cite{MOT};

\item There is a Heegaard Floer invariant for three-manifolds with boundary, called bordered Floer homology \cite{LOT}. If we decompose a three-manifold $Y$ along a surface, $\HFhat(Y)$ can be recovered as the tensor product of the bordered invariants of the two pieces. There is also an extension of this theory to $3$-manifolds with codimension $2$ corners \cite{DMcornered, DLMcornered};

\item Using bordered Floer homology, one can give an effective algorithm for computing $\HFhat$ for  $3$-manifolds \cite{LOTHF}. Further, one can compute $\HFhat$ in infinite families, for example for graph manifolds \cite{Hanselman}.
\end{itemize}

Next, let us discuss several useful properties of Heegaard Floer homology; when combined with the calculational techniques above, they lead to many topological applications. 

Some properties were inspired by the corresponding ones in gauge theory (for instanton and/or monopole Floer homology):
\begin{enumerate}[(i)] \itemsep5pt
\item A cobordism between $3$-manifolds (together with a $\spinc$ structure on that cobordism) induces a map between the respective Heegaard Floer homologies \cite{HolDiskFour}. One can also define invariants of closed $4$-manifolds, which behave similarly (and are conjecturally identical) to the Seiberg-Witten invariants;

\item There is a surjective homomorphism $d: \Theta_H^3 \to \Z$ given by
$$ d(Y) = \min \{r \mid \exists \ x \neq 0 \text{ s.t. } \forall l, \ x \in \operatorname{im}(U^l: \HFp_{r+2l}(Y) \to \HFp_r(Y)) \},$$
and we have the analog of the inequality \eqref{eq:Fineq}; cf. \cite{AbsGraded}. (Note that $d$ corresponds to $2\delta$.) More generally, we can define $d(Y, \s)$ for any rational homology sphere and $\s \in \spinc(Y)$;

\item Heegaard Floer homology detects the Thurston norm of $3$-manifolds (which gives the minimal complexity of surfaces in a given homology class) \cite{GenusBounds};

\item A contact structure $\xi$ on $Y$ induces an element $c(\xi) \in \HFhat(-Y)/\pm 1$ \cite{HolDiskContact}. We have $c(\xi)=0$ for overtwisted contact structures, and $c(\xi)\neq 0$ for symplectically semi-fillable contact structures \cite{GenusBounds};

\item We say that a rational homology sphere $Y$ is an L-space if $\HFhat(Y, \s) = \Z$ for any $\s \in \spinc(Y)$. If $Y$ is an L-space, then $Y$ does not admit a co-orientable taut foliation  \cite{GenusBounds}.
\end{enumerate}

Other properties of Heegaard Floer homology were developed first in this setting (and sometimes  inspired similar results in gauge theory). This is the case with fiberedness detection \cite{Ghiggini, NiFibered, JuhaszDecompose, Ni3}, with surgery formulas \cite{IntSurg, RatSurg, LinkSurg}, and with extending Heegaard Floer homology to knots \cite{Knots, RasmussenThesis}, sutured $3$-manifolds \cite{Juhasz}, and bordered $3$-manifolds \cite{LOT}. 

\medskip

We now list a few concrete topological applications of Heegaard Floer homology.
\medskip

By making use of the contact invariant $c(\xi)$, one can study tight contact structures on various classes of $3$-manifolds. For example: 
\begin{theorem}[Lisca-Stipsicz \cite{LiscaStipsicz2, LiscaStipsicz}] 
A closed, oriented, Seifert fibered $3$-manifold $Y$ admits a positive tight contact structure if and only if $Y$ is not diffeomorphic to $(2n-1)$ surgery on the torus knot $T_{2,2n+1}$ for any $n \geq 1$.
\end{theorem}

Using the fiberedness and genus detection properties of knot Floer homology, one gets: 
\begin{theorem}[Ghiggini \cite{Ghiggini}]
If $K \subset S^3$ and $r \in \Q$ are such that $S^3_r(K)$ is the Poincar\'e sphere, then $K$ is the trefoil.
\end{theorem}

By combining Ghiggini's methods with the surgery formula from \cite{RatSurg}, one obtains a surgery characterization (an analog of Theorem~\ref{thm:KMOS}) for a few non-trivial knots:
\begin{theorem}[Ozsv\'ath-Szab\'o \cite{SurgeryTE}]
Let $K$ be the left-handed trefoil, the right-handed trefoil, or the figure-eight knot. Suppose $K' \subset S^3$ is a knot such that there is an orientation-preserving diffeomorphism $S^3_r(K) \cong S^3_r(K')$, for some $r \in \Q$. Then $K=K'$.
\end{theorem}

Using $d$ invariants and surgery formulas, one gets constraints on the knots in $S^3$ that can produce lens spaces by surgery:
\begin{theorem}[Ozsv\'ath-Szab\'o \cite{OSLens}]
If $K \subset S^3$ is such that $S^3_r(K)$ is a lens space for some $r \in \Q$, then the Alexander polynomial of $K$ is of the form
$$\Delta_K(q) = \sum_{j=-k}^k (-1)^{k-j} q^{n_j},$$
for some $k \geq 0$ and integers $n_{-k} < \dots < n_k$ such that $n_{-j} = - n_j$.
\end{theorem}
Combining this with fiberedness detection \cite{NiFibered}, one obtains the additional constraint that $K$ is fibered (under the same hypotheses).

If surgery on a knot $K$ gives a lens space, one can also obtain inequalities between the surgery slope and the genus of the knot, $g(K)$. For example:
\begin{theorem}[Rasmussen \cite{RasmussenGT}]
Let $K \subset S^3$ be a knot such that $S^3_r(K)$ is a lens space, for some $r \in \Q$. Then:
$$ |r| \leq 4g(K)-3.$$
\end{theorem}

\begin{theorem}[Greene \cite{Greene}]
\label{thm:greene}
Suppose that $K \subset S^3$ is a knot such that $S^3_p(K)$ is a lens space for some positive integer $p$. Then:
$$ 2g(K)-1 \leq p-2\sqrt{(4p+1)/5}$$
unless K is the right-hand trefoil and p = 5.
\end{theorem}

The proof of Theorem~\ref{thm:greene} combines methods based on the $d$ invariant with Donaldson's diagonalizability theorem. Similar techniques allowed Greene to give a complete characterization of which lens spaces can be obtained by integral surgery on a knot in $S^3$.

By using the rational surgery formula from \cite{RatSurg}, one can study cosmetic surgeries, that is, surgeries (with different coefficients) on the same knot, that produce the same $3$-manifold. Building up on work of Ozsv\'ath-Szab\'o \cite{RatSurg}, Ni and Wu proved:
\begin{theorem}[Wu \cite{Wu}; Ni-Wu, \cite{NiWu}]
Suppose $K \subset S^3$ is a non-trivial knot such that $S^3_{r_1}(K) \cong S^3_{r_2}(K)$ (as oriented manifolds) for two distinct rational numbers $r_1$ and $r_2$. Then, $r_1 = -r_2$ and $r_1$ is of the form $p/q$ where $p$, $q$ are coprime integers with $q^2 \equiv -1\pmod{p}.$
\end{theorem}

Using $d$ invariants, one can show that various $3$-manifolds with $b_1=1$ are not surgery on a knot in $S^3$ \cite{AbsGraded}. By more refined methods (based on the knot surgery formulas), one can show that certain families of integer homology spheres are not surgeries on knots. For example:
\begin{theorem}[Hom-Karakurt-Lidman \cite{HKL}]
For $k \geq 4$, the Brieskorn spheres $\Sigma(2k, 4k-1, 4k+1)$ are not surgeries on knots in $S^3$.
\end{theorem}

\section{Floer stable homotopy}
\label{sec:swf}
Suppose we have an infinite dimensional space $X$ and a function $f: X \to \R$ so that we can define some variant of Floer homology $\HF(X, f)$. In \cite{CJS}, Cohen, Jones and Segal asked the following question: Can $\HF(X,f)$ be expressed as the homology of a ``Floer space'' $\S(X, f)$? They proposed a construction along the following lines: We choose an absolute grading on the Floer complex, lifting the relative grading. Then, to each generator of the Floer complex in degree $k$ we associate a $k$-cell; this is attached to the lower dimensional cells by maps determined by the spaces of gradient flow lines, according to the Pontrjagin-Thom construction. 

Let us illustrate this by an example: Suppose the Floer complex has only  two generators $x$ and $y$, with relative index $\mu(x, y) = k \geq 1$. The space of flow lines between $x$ and $y$ is an $k$-dimensional manifold, with an action of $\R$ by translation. Dividing by this action we obtain a $(k-1)$-dimensional manifold $P$. Under certain hypotheses, $P$ is closed, and can be equipped with a stable framing (a stable normal trivialization). If so, then the Pontrjagin-Thom construction produces an element in the stable homotopy group of spheres $\pi_{k-1}^{st}(S^0)$, represented by a map 
$$ \rho : S^{N+k-1} \to S^N$$
for $N \gg 0$. The desired Floer space $\S(X, f)$ is obtained from an $N$-cell and an $(N+k)$-cell, with the attaching map being $\rho$. 

There are several caveats about this construction:
\begin{enumerate}[(i)] \itemsep5pt
\item If we increase $N$, then the space changes by a suspension. Thus, it makes more sense to define $\S(X, f)$ as a stable homotopy type (suspension spectrum);

\item In many cases, the Floer complex has infinitely many generators, in infinitely many degrees.  Cohen, Jones and Segal propose that in such situations the natural object to define is a pro-spectrum (an inverse system of spectra);

\item The spaces $P$ of flow lines may not be compact, for two reasons: bubbling (which happens in instanton and in Lagrangian Floer theory, but not in monopole theory), and the presence of degenerations of flow lines into broken flow lines. If we assume no bubbling, then the spaces $P$ are expected to be manifolds-with-corners, which can be put together into attaching maps in a manner discussed in \cite{CJS};

\item Even if the non-compactness issues are resolved, we still need to specify stable framings for the Pontrjagin-Thom construction. Cohen, Jones and Segal identify a class in $KO^1(X)$ that obstructs the existence of such framings; 

\item Even if the obstruction is zero, to define the framings we need to endow the spaces of flow lines with smooth (not just topological) structures of manifolds-with-corners, such that these structures are compatible with each other. This leads into some difficult analytical issues.
\end{enumerate}

In the case of Seiberg-Witten Floer homology on $3$-manifolds $Y$ with $b_1(Y)=0$, a couple of the problems above disappear: There are only finitely many generators (so we expect a spectrum, rather than a pro-spectrum), there is no bubbling, and the framing obstruction vanishes. Still, defining the smooth manifold-with-corners structures seems difficult. 

A way of going around this problem was developed in \cite{Spectrum}. Rather than follow the Cohen-Jones-Segal program, one applies Furuta's technique of finite dimensional approximation \cite{Furuta}.  The configuration space $X$ of connections and spinors is a Hilbert space. We choose a certain sequence of finite dimensional subspaces $X_{\lambda}$ that are getting larger as $\lambda \to \infty$, so that their union is dense in $X$. We consider an approximate Seiberg-Witten flow on $X_{\lambda}$. Of course, on a closed finite dimensional manifold, instead of Morse homology we can simply take the singular homology and get the same answer. Our vector spaces $X_{\lambda}$ are non-compact, but a similar procedure works: We consider the Conley index \cite{ConleyBook} associated to the flow on a large ball $B \subset X_{\lambda}$. Roughly, the Conley index is the pointed space
$$ I_{\lambda} =  B / L$$
where $L \subset \del B$ is the part of the boundary of $B$ where the flow goes outwards. The homology of the Conley index is meant to be the Morse homology associated to the approximate flow (assuming that the flow is Morse-Smale). 

In \cite{Spectrum}, we do not need to assume the Morse-Smale transversality condition. Rather, we define Seiberg-Witten Floer homology directly as the relative homology of $I_{\lambda}$, with an appropriate degree shift depending on $\lambda$. This yields the benefit that we also get a Floer stable homotopy type, the suspension spectrum associated to $I_{\lambda}$. Since the Seiberg-Witten equations have an $S^1$ symmetry, we actually have an $S^1$-equivariant stable homotopy type
$$ \swf(Y, \s)$$
associated to a rational homology sphere $Y$ and a $\spinc$ structure $\s$ on $Y$.

Starting from here, if $h$ is a generalized homology theory (such as K- or KO-theory, complex bordism, stable homotopy, etc.), one can define a Seiberg-Witten Floer generalized homology:
$$ h_*(\swf(Y, s)).$$

This turns out to be particularly useful when combined with additional symmetry of the Seiberg-Witten equations, the conjugation symmetry. Let us focus on the case when $Y$ is a homology sphere, so that there is a unique $\spinc$ structure $\s$, coming from a spin structure. The conjugation and the $S^1$ symmetry together yield a symmetry by the group 
$$ \pin = S^1 \oplus S^1  j \subset \C \oplus \C j = \H,$$
where $\H$ are the quaternions and $j^2=-1$. We can then define $\swf(Y) =\swf(Y, \s)$ as a $\pin$-equivariant stable homotopy type \cite{beta}, and for example take its equivariant (Borel) homology
\begin{equation}
\label{eq:hfpin}
 \swfh^{\pin}_*(Y)= \tH^{\pin}_*(\swf(Y)).
 \end{equation}

This is the {\em Pin(2)-equivariant Seiberg-Witten Floer homology} of $Y$. In Section~\ref{sec:pin2} we will describe its application to the resolution of the triangulation question in high dimensions.

One can also define {\em Pin(2)-equivariant Seiberg-Witten Floer $K$-theory} by 
$$ \swfk^{\pin}(Y)= \tK^{\pin}(\swf(Y)).$$

This has applications to the topology of four-manifolds with boundary \cite{kg, FurutaLi}. They are inspired from Furuta's proof of the $10/8$ inequality for closed, smooth, spin four-manifolds: If $W$ is such a manifold, Furuta showed that
$$b_2(W) \geq \frac{10}{8}|\sigma(W)| + 2,$$
where $\sigma$ denotes the signature. (Matsumoto's $11/8$ conjecture \cite{Matsumoto} postulates the stronger inequality $b_2(W) \geq \tfrac{11}{8} |\sigma(W)|$.)

Now suppose that $W$ is a smooth, spin, compact $4$-manifold with boundary a homology sphere $Y$. From $\swfk^{\pin}(Y)$ one can extract an invariant $\kappa(Y) \in \Z$, and then prove an analog of Furuta's inequality:
$$b_2(W) \geq \frac{10}{8}|\sigma(W)| + 2 - 2\kappa(Y).$$

Slightly stronger inequalities can be obtained by considering $\pin$-equivariant KO-theory instead of K-theory; see \cite{LinKO}.

\section{The triangulation conjecture}
\label{sec:pin2}

A triangulation of a topological space is a homeomorphism to a simplicial complex.  In 1924, Kneser \cite{Kneser} asked the following:
\begin{question}
\label{q2}
Does every topological manifold admit a triangulation?
\end{question}

The answer was initially thought to be positive, and this was called the (simplicial) triangulation conjecture. A stronger version of this was the combinatorial triangulation conjecture, which posited that manifolds admit triangulations such that the links of the simplices are spheres. Such triangulations are called combinatorial, and are equivalent to PL (piecewise linear) structures on those manifolds.

Here is a short history of the relevant developments:

\begin{itemize} \itemsep5pt
\item Rad\'o \cite{Rado} proved that two-dimensional manifolds admit combinatorial triangulations;
\item Cairns \cite{Cairns} and Whitehead \cite{Whitehead} showed the same for smooth manifolds, of any dimension; 
\item Moise \cite{Moise} showed that three-manifolds have combinatorial triangulations;
\item Kirby and Siebenmann \cite{KSbook} showed that the combinatorial triangulation conjecture is false: There exist manifolds without PL structures in every dimension $\geq 5$. Further, they showed that in these dimensions, the existence of PL structures is governed by an obstruction class $\Delta(M) \in H^4(M; \Z/2)$;
\item Edwards \cite{Edwards} gave the first example of a non-combinatorial triangulation of a manifold: the double suspension of a certain homology $3$-sphere is homeomorphic to $S^5$, but the underlying triangulation is non-combinatorial; 
\item Freedman \cite{Freedman} found $4$-dimensional manifolds without PL structures, e.g., the $E_8$-manifold; 
\item Casson \cite{Casson} proved that, for example, Freedman's $E_8$-manifold does not admit any triangulations. This gave the first counterexamples to the simplicial triangulation conjecture (in dimension $4$);
\item The simplicial triangulation question in dimension $\geq 5$ was shown by Galewski-Stern \cite{GS} and Matumoto \cite{Matumoto} to be equivalent to a different problem in $3+1$ dimensions. This problem was solved in \cite{beta}, using $\pin$-equivariant Seiberg-Witten Floer homology. As a consequence, there exist non-triangulable manifolds in any dimension $\geq 5$.
\end{itemize}

Let us sketch the disproof of the triangulation conjecture in dimensions $\geq 5$. 

Suppose that a closed, oriented $n$-dimensional manifold $M$ ($n \geq 5$) is equipped with a triangulation $K$. Consider the Sullivan-Cohen-Sato class (cf. \cite{Sullivan, Cohen, Sato}):
\begin{equation}
\label{eq:cK}
 c(K) = \sum_{\sigma \in K^{(n-4)}} [\link_K(\sigma)] \cdot \sigma \in H_{n-4}(M; \Theta^H_3) \cong H^4(M; \Theta^H_3).
 \end{equation}
Here, the sum is taken over all codimension four simplices in the triangulation $K$. The link of each such simplex can be shown to be a homology $3$-sphere. (It would be an actual $3$-sphere if the triangulation were combinatorial.) Note the appearance of the homology cobordism group $\Theta^H_3$ defined in \eqref{eq:theta}. We focus on codimension four simplices in \eqref{eq:cK}, because the analog of the homology cobordism group in any other dimension is trivial \cite{Kervaire}.

The Rokhlin homomorphism $\mu$ from \eqref{eq:mu} induces a short exact sequence
\begin{equation}
\label{eq:ses}
 0 \To \ker(\mu) \To \Theta^H_3 \To \Z/2 \To 0
 \end{equation}
and an associated long exact sequence in cohomology
\begin{equation}
\dots \To H^4(M; \Theta^H_3) \xrightarrow{\hspace*{2pt} \mu_* \hspace*{2pt}} H^4(M; \Z/2)  \xrightarrow{\hspace*{2pt} \delta \hspace*{2pt}} H^5(M; \ker(\mu)) \To \dots.
\end{equation}

It can be shown that the image of $c(K)$ under $\mu_*$ is exactly the Kirby-Siebenmann obstruction to PL structures, $\Delta(M) \in H^4(M; \Z/2)$. Thus, if $M$ admits any triangulation, we get that $\Delta(M)$ is in the image of $\mu_*$, and hence in the kernel of the Bockstein homomorphism $\delta$. Thus, a necessary condition for the existence of simplicial triangulations is the vanishing of the class
$$ \delta(\Delta(M)) \in  H^5(M; \ker(\mu)).$$

Interestingly, this is also a sufficient condition:

\begin{theorem}[Galewski-Stern \cite{GS}; Matumoto \cite{Matumoto}]
A topological manifold $M$ of dimension $\geq 5$ is triangulable if and only if $\delta(\Delta(M)) =0$.
\end{theorem}

Thus, we need to find out if there exist manifolds $M$ with $\delta(\Delta(M))\neq 0$. Observe that the Bockstein map $\delta$ is zero if the short exact sequence \eqref{eq:ses} splits. Thus, if \eqref{eq:ses} split, then all high dimensional manifolds would be triangulable. The converse is also true:

\begin{theorem}[Galewski-Stern \cite{GS}; Matumoto \cite{Matumoto}]
\label{thm:gsm}
There exist non-triangulable manifolds of (every) dimension $\geq 5$ if and only if the exact sequence \eqref{eq:ses} does not split.
\end{theorem}

\begin{example} (due to Peter Kronheimer)
By Freedman's theorem \cite{Freedman}, simply connected, closed topological four-manifolds are characterized up to homeomorphism by their intersection form and their Kirby-Siebenmann invariant. Let $W$ be the fake $\cp^2 \# (-\cp^2)$, that is, the closed, simply connected topological $4$-manifold with intersection form $Q=\langle 1 \rangle \oplus \langle -1 \rangle$ and with non-trivial Kirby-Siebenmann invariant. Since the form $Q$  is isomorphic to $-Q$, by applying Freedman's theorem again we find that $W$ admits an orientation-reversing homeomorphism $f: W \to W$. Let $M$ be the mapping torus of $f$. Then $M$ is a five-manifold with the Steenrod square $\Sq^1 \Delta(M) \in H^5(M; \Z/2)$ non-trivial. Assuming that \eqref{eq:ses} does not split, it is not hard to see that the non-vanishing of $\Sq^1 \Delta(M)$ implies the non-vanishing of $\delta(\Delta(M))$. Therefore, $M$ is non-triangulable. By taking products with the torus $T^{n-5}$, we obtain non-triangulable manifolds in any dimension $n \geq 5$.
\end{example}

In view of Theorem~\ref{thm:gsm}, the disproof of the triangulation conjecture is completed by the following:

\begin{theorem}[Manolescu \cite{beta}]
\label{thm:beta}
The short exact sequence \eqref{eq:ses} does not split.
\end{theorem} 

\begin{proof}[Sketch of the proof] A splitting of \eqref{eq:ses} would consist of a map $\eta: \Z/2 \to \Theta^H_3$ with $\mu\circ \eta=\id$; that is, there would be a homology $3$-sphere $Y$ such that $\mu(Y)=1$ and $2[Y]=0\in \Theta^H_3$. 

To show that such a sphere does not exist, we construct a lift of $\mu$ to the integers,
$$ \beta : \Theta^H_3 \to \Z,$$
with the following properties:
\begin{enumerate}[(a)]
\item If $-Y$ denotes $Y$ with the orientation reversed, then $\beta(-Y) = - \beta(Y)$;
\item The mod $2$ reduction of $\beta(Y)$ is the Rokhlin invariant $\mu(Y)$.
\end{enumerate}

Given such a $\beta$, if we had a  homology sphere $Y$ of order two in $\Theta^H_3$, then $Y$ would be homology cobordant to $-Y$, and we would obtain
$$ \beta(Y) = \beta(-Y) = - \beta(Y),$$
hence $\beta(Y)=0$ and therefore $\mu(Y) = 0$. 

It remains to construct $\beta$. Its definition is modeled on that of the Fr{\o}yshov invariant $\delta$ from \eqref{eq:delta}, but instead of the ($S^1$-equivariant) monopole Floer homology $\HMto$, we use the $\pin$-equivariant Seiberg-Witten Floer homology $\swfh^{\pin}$ from \eqref{eq:hfpin}. 

Specifically, we consider $\swfh^{\pin}_*(Y)$ with coefficients in the field $\F = \Z/2.$ It is a module over the ring
$$ H^*_{\pin}(pt; \F)= H^*(B\pin; \F) = \F[q, v]/(q^3),$$
where $q$ is in degree $1$ and $v$ is in degree $4$. Then, we set
$$ b(Y)=\min \{r\equiv 2\mu(Y) +1 (\text{mod }{4}), \exists \ x \in \swfh^{\pin}_r(Y), 0\neq x \in \operatorname{im}(v^l), \forall l\}$$
and then normalize this to
$$ \beta(Y) =\tfrac{1}{2}(b(Y)-1).$$

Property (a) of $\beta$ and the fact that $\beta$ descends to a map on $\Theta^H_3$ are similar to what happens for the Fr{\o}yshov invariant, and can be proved in a similar manner.

More interesting is property (b) for $\beta$, which is satisfied because by construction we asked that $b(Y)\equiv 2\mu(Y) + 1 (\text{mod }{4})$. However, one needs to show that $\swfh^{\pin}(Y)$ contains nonzero elements $x$ in degrees congruent to $2\mu(Y) +1$ mod $4$, and such that they are in the image of $v^l$ for all $l$. 

To get an idea for why this is true, it is helpful to imagine that $\swfh^{\pin}(Y)$ is the homology of a complex generated by solutions to the Seiberg-Witten equations on $Y$ (although its actual definition from Section~\ref{sec:swf} is in terms of the singular homology of the Conley index). The Seiberg-Witten equations have some irreducible solutions (on which the group $\pin$ acts freely), and each such $\pin$ orbit contributes a copy of $\F$ to the chain complex. There is also a unique reducible solution, on which $\pin$ acts trivially, and which contributes a copy of $H_*^{\pin}(pt; \F)= H_*(B\pin; \F)$ to the complex. Further, the bottom degree element in $H_*(B\pin; \F)$ coming from the reducible is in a degree congruent to $2\mu(Y)$ mod $4$. (This is standard in Seiberg-Witten theory, and follows from a relation between eta invariants and the Rokhlin homomorphism.) The homology $H_*(B\pin; \F)$ (and the cap product action on it by the cohomology of $B\pin$) can be depicted as follows:
\begin{equation}
 \xymatrixcolsep{.7pc}
\xymatrix{
 \F  &  \F \ar@/_1pc/[l]_{q} &  \F \ar@/_1pc/[l]_{q} & 0 & \F \ar@/^1pc/[llll]^{v} & \F \ar@/_1pc/[l]_{q} \ar@/^1pc/[llll]^{v} & \F \ar@/_1pc/[l]_{q} \ar@/^1pc/[llll]^{v} & 0 & \dots  \ar@/^1pc/[llll]^{v} & \dots \ar@/^1pc/[llll]^{v} & \dots \ar@/^1pc/[llll]^{v}
} 
\end{equation}

Thus, there are three infinite $v$-tails, which live in degrees congruent to $2\mu(Y), 2\mu(Y)+1$ and $2\mu(Y)+2$ mod $4$. Since there are only finitely many irreducibles, their interaction with the tails in the chain complex is limited to some degree range. It follows that there must be some element in each of these tails that survives in homology. To define $\beta$ we focus on the middle tail. The other two tails produce maps $\alpha, \gamma : \Theta^H_3 \to \Z$ that do not quite satisfy the desired property (a) under orientation reversal; rather, we have
$$ \alpha(-Y)=-\gamma(Y).$$

On the other hand, $\beta$ satisfies both properties (a) and (b).

It is worth explaining why the same argument does not work in the case of the $S^1$-equivariant Seiberg-Witten Floer homology (which corresponds to $\HMto$ from Section~\ref{sec:mon}). That homology is a module over the ring $\Z[U]$ with $U$ in degree $2$, and the reducible contributes a copy of $H_*(\cp^{\infty})$ to the Floer complex. The bottom element is again in a degree congruent to $2\mu(Y)$ mod $4$. However, when we pass to homology, the new bottom element (which is used to define the Fr{\o}yshov invariant) may no longer have the same grading mod $4$. This boils down to the fact that $H_*(\cp^{\infty})$ is $2$-periodic, whereas $H_*(B\pin; \F)$ is $4$-periodic.

Let us illustrate this with an example: the Brieskorn sphere $Y=\Sigma(2,3,11)$, equipped with a suitable metric. There is one $\pin$-orbit of irreducible solutions to the Seiberg-Witten equations, in degree $1$. The reducible solution is in degree $0$, and indeed we have $\mu(\Sigma(2,3,11))=0$. There are flow lines from the irreducibles to the reducible, which contribute to the Floer differential. 
Precisely, the $\pin$-equivariant Seiberg-Witten Floer complex of $\Sigma(2,3,11)$ is
\begin{equation}
\label{eq:pin11c}
 \xymatrixcolsep{.7pc}
\xymatrixrowsep{0pc}
\xymatrix{
  \F  &  \F \ar@/_1pc/[l]_{q} &  \F \ar@/_1pc/[l]_{q} & 0 & \F \ar@/^1pc/[llll]^{v} & \F \ar@/_1pc/[l]_{q} \ar@/^1pc/[llll]^{v} & \F \ar@/_1pc/[l]_{q} \ar@/^1pc/[llll]^{v} & 0 & \dots  \ar@/^1pc/[llll]^{v} & \dots \ar@/^1pc/[llll]^{v} & \dots \ar@/^1pc/[llll]^{v}\\
   & \oplus  & & & & & & & & & \\
    & \F\ar@/^1pc/[uul]^{\del} & & & & & & & & & 
} \end{equation}
with the leftmost element in degree $0$. Its homology is
\begin{equation}
\label{eq:pin11h}
 \xymatrixcolsep{.7pc}
\xymatrixrowsep{0pc}
\xymatrix{
{\phantom{\F}} &  \F  &  \F \ar@/_1pc/[l]_{q} & 0 & \F  & \F \ar@/_1pc/[l]_{q} \ar@/^1pc/[llll]^{v} & \F \ar@/_1pc/[l]_{q} \ar@/^1pc/[llll]^{v} & 0 & \dots  \ar@/^1pc/[llll]^{v} & \dots \ar@/^1pc/[llll]^{v} & \dots \ar@/^1pc/[llll]^{v} 
} \end{equation}
with the leftmost element in degree $1$. We obtain $b(Y)=1$, so $\beta(Y)=0$, in agreement with $\mu(Y)=0$.

By contrast, the $S^1$-equivariant Seiberg-Witten Floer complex of $\Sigma(2,3,11)$ is
\begin{equation}
\label{eq:11c}
 \xymatrixcolsep{.7pc}
\xymatrixrowsep{0pc}
\xymatrix{
  \Z  &  0 & \Z \ar@/_1pc/[ll]_{U} & 0 & \Z \ar@/_1pc/[ll]_{U} & 0 & \dots \ar@/_1pc/[ll]_{U} \\
& \oplus  & & & & & &  \\
  \ \ \  & \Z\ar[uul]^{\del} & & & & & & \\
 & \oplus & & & & & &  \\
\ \ \  & \Z\ar@/^1pc/[uuuul]^{\del} & & & & & &  
} \end{equation}
with the leftmost element in degree $0$. Note that the $\pin$ orbit consists of two $S^1$ orbits, which 
produce the two copies of $\Z$ at the bottom. The $S^1$-equivariant Seiberg-Witten Floer homology is
\begin{equation}
\label{eq:11h}
 \xymatrixcolsep{.7pc}
\xymatrixrowsep{0pc}
\xymatrix{
{\phantom{\Z}} &  0 & \Z  & 0 & \Z \ar@/_1pc/[ll]_{U} & 0 & \dots \ar@/_1pc/[ll]_{U} \\
& \oplus  & & & & & &  \\
  \ \ \  & \Z & & & & & & \\
 } \end{equation}
with the bottom $0 \oplus \Z$ in degree $1$. From here we get $\delta(Y)=2/2=1$, which no longer gives $\mu(Y)$ modulo $2$.
\end{proof}

A different construction of $\pin$-equivariant Seiberg-Witten Floer homology was given by Lin in \cite{FLin}. Rather than doing finite dimensional approximation, Lin extends the Kronheimer-Mrowka definition of monopole Floer homology \cite{KMBook} to a Morse-Bott setting, which is suitable for preserving the $\pin$-equivariance of the equations. One can give an alternate disproof of the triangulation conjecture using Lin's construction; see \cite{FLin} for more details.

\bibliographystyle{amsalpha}
\bibliography{biblio}

\end{document}